\documentclass[12pt]{amsart}
\usepackage{amsmath,amssymb,hyperref,geometry,graphicx,tikz,ytableau,fp,mathdots,amsthm}
\usetikzlibrary{cd}

\newcounter{intro}[section]

\newtheorem{result}[intro]{Result}

\newtheorem{theorem}[equation]{Theorem}

\newtheorem{proposition}[equation]{Proposition}

\theoremstyle{definition}

\newtheorem{remark}[equation]{Remark}

\newcommand{\M}{\overline{\mathcal{M}}}
\newcommand{\ps}{\mathrm{ps}}
\newcommand{\T}{\mathcal{T}}
\newcommand{\bE}{\mathbb{E}}
\newcommand{\cE}{\mathcal{E}}
\newcommand{\G}{\mathcal{G}}
\renewcommand{\O}{\mathcal{O}}

\newcommand{\C}{\mathcal{C}}
\newcommand{\ch}{\mathrm{ch}}
\renewcommand{\L}{\mathbb{L}}
\newcommand{\Q}{\mathbb{Q}}
\newcommand{\N}{\mathbb{N}}
\newcommand{\td}{\mathrm{td}}
\newcommand{\Z}{\mathbb{Z}}

\numberwithin{equation}{section}

\begin{document}

\title{Pseudostable Hodge integrals}
\author[R.~Cavalieri]{Renzo Cavalieri}
\address{Department of Mathematics, Colorado State University}
\email{renzo@math.colostate.edu }

\author[J.~Gallegos]{Joel Gallegos}
\address{Department of Mathematics, University of Minneapolis}
\email{galle268@umn.edu}

\author[D.~Ross]{Dustin Ross}
\address{Department of Mathematics, San Francisco State University}
\email{rossd@sfsu.edu}

\author[B.~Van Over]{Brandon Van Over}
\address{Department of Mathematics, Colorado State University }
\email{bvanover@colostate.edu }

\author[J.~Wise]{Jonathan Wise}
\address{Department of Mathematics,  University of Colorado}
\email{jonathan.wise@math.colorado.edu}

\begin{abstract}
This paper initiates a study of Hodge integrals on moduli spaces of pseudostable curves. We prove an explicit comparison formula that allows one to effectively compute any pseudostable Hodge integral in terms of intersection numbers on moduli spaces of stable curves, and we use this comparison to prove that pseudostable Hodge integrals are equal to their stable counterparts when they are linear in lambda classes, but not when they are nonlinear. This suggests that pseudostable Gromov--Witten invariants are equal to usual Gromov--Witten invariants for target curves, but not for higher-dimensional target varieties.
\end{abstract}

\maketitle

\section{Introduction}

Ever since their introduction by Deligne and Mumford \cite{DeligneMumford}, the moduli spaces of stable curves, denoted $\M_{g,n}$, have played a central role in algebraic geometry, with applications ranging from enumerative geometry to theoretical physics. These moduli spaces support special Chow classes $\lambda_1,\dots,\lambda_g,\psi_1,\dots,\psi_n\in A^*(\M_{g,n})$, and \emph{Hodge integrals} are the intersection numbers associated to any polynomial in these classes. Hodge integrals exhibit a great deal of beautiful structure (see, for example, \cite{FP}), and due to virtual degeneration and localization techniques (\cite{BehrendFantechi,Li,GraberPandharipande}), nearly all concrete computations in Gromov--Witten theory can be reduced to computations of Hodge integrals.

A variation on stable curves, called pseudostable curves, was introduced by Schubert \cite{Schubert}, and has played an important role in the minimal model program for $\M_{g,n}$ (\cite{HH}). Roughly speaking, pseudostability differs from usual stability in that curves are allowed to have cusp singularities but are not allowed to have elliptic tails. Analogous to the setting of stable curves, the moduli spaces of pseudostable curves also support special Chow classes $\lambda_1,\dots,\lambda_g,\psi_1,\dots,\psi_n\in A^*(\M_{g,n}^\ps)$, and pseudostable Hodge integrals are the intersection numbers associated to any polynomial in these classes. These intersection numbers have not previously been studied, and \emph{the purpose of this paper is to initiate a systematic study of pseudostable Hodge integrals, with the aim of uncovering the structure inherent to these integrals and studying the implications of this structure on pseudostable Gromov--Witten theory. }

\subsection{Statement of results}

Our main result provides an effective means by which one can compute any pseudostable Hodge integral in terms of intersection numbers on moduli spaces of stable curves. To set up notation, let
\[
\G^k:\M_{g-k,n+k}\times\M_{1,1}^{\times k}\rightarrow \M_{g,n}
\]
be the natural gluing map that sends $(C_0, C_1, \ldots, C_k)$ to the curve obtained by attaching $C_1, \ldots, C_k$ to $C_0$ at the marked points labelled $n+1, \ldots, n+k$. Let $p_0$ be the projection map to the first factor of the domain. 

\begin{result}[Theorem~\ref{thm:translation}]
For any polynomial $F\in\Q[x_1,\dots,x_g,y_1,\dots,y_n]$, we have
\[
\int_{\M_{g,n}^\ps}F(\lambda_1,\dots,\lambda_g,\psi_1,\dots,\psi_n)=\int_{\M_{g,n}}F(\hat\lambda_1,\dots,\hat\lambda_g,\psi_1,\dots,\psi_n)
\]
where
\[
\hat\lambda_j=\lambda_j+\sum_{i=1}^j\frac{1}{i!}\G_*^{i}(p_0^*(\lambda_{j-i})).
\]	
\end{result}

The integrals appearing in the right-hand side of Result~A can be computed in terms of usual Hodge integrals (see, for example, \cite{Yang}), or they can be evaluated using any of the various computer programs that have been created for computing this type of intersection number on $\M_{g,n}$ (see, for example, \cite{DSvZ}). 

Employing Result~A, our second main result shows that any pseudostable Hodge integral that is linear in $\lambda$ classes is equal to its stable counterpart.

\begin{result}[Proposition~\ref{thm:linearhodge}]
For any $j=1,\dots,g$ and any polynomial $F\in\Z[x_1,\dots,x_n]$,
\[
\int_{\M_{g,n}^\ps}\lambda_jF(\psi_1,\dots,\psi_n)=\int_{\M_{g,n}}\lambda_jF(\psi_1,\dots,\psi_n).
\]
\end{result}

When we first began this project, Result B was entirely unexpected, and we were quite surprised to discover computationally that all of the correction terms from an earlier version of Result~A summed to zero. This computational realization is what led us to the concise formulation of Result~A presented above. In addition to proving Result~B as a consequence of Result~A, we also describe a more conceptual understanding of why it is true (see Remark~\ref{rmk:ELSV}), but this conceptual understanding depends on a number of foundational tools in psuedostable maps that have not yet been developed rigorously.

One might be so optimistic as to hope that \emph{all} pseudostable Hodge integrals are equal to their stable counterparts; however, we prove that this is not the case (Proposition~\ref{thm:mumford}), while simultaneously showing that Mumford's formula, which is a key computational tool in Gromov--Witten theory, fails in the pseudostable setting.

\subsection{Context, motivation, and future work} Throughout the last three decades, a sizeable body of work in mathematics and theoretical physics has centered around the investigation of \emph{curve counting theories} (see \cite{PT} for an introductory overview). These developments have had far-reaching implications, from solutions to classical problems in enumerative geometry to advances in string theory. Given a variety $X$, the basic ingredient required to count curves in $X$ is a suitably well-behaved moduli space that parametrizes ``curves'' in $X$. There are various ways that one might interpret what it means to be a ``curve'' in $X$; each interpretation leads to a different moduli space and, thus, a different curve-counting theory.

One of the earliest-developed curve-counting theories is Gromov--Witten theory, where the relevant moduli space parametrizes stable maps to $X$. In this setting, the curves that one considers are very nice---they have at worst nodal singularities---but the maps that insert these curves into $X$ can be quite ill-behaved---for example, stable maps might contract entire irreducible components of curves. At the other end of the spectrum, Donaldson--Thomas theory studies Hilbert schemes of curves in $X$. In this setting, the curves are much more complicated---they can be highly singular and have nonreduced scheme structure---but the maps that insert them in $X$ are as nice as can be---they are embeddings.

The philosophical motivation for this work is to start with Gromov--Witten theory, then to investigate what happens to the curve counts if we take a small step in the direction of Donaldson--Thomas theory. In other words, starting with stable maps, we ask: What happens to the corresponding curve-counting theory if we allow the curves to be a little more singular while requiring the maps to be a little better behaved? In particular, allowing the curves to have cusp singularities and disallowing the maps to contract elliptic tails leads to the notion of pseudostable maps, generalizing the notion of pseudostable curves.

So where do Hodge integrals come into play? Foundational results in Gromov--Witten theory tell us that spaces of stable maps support virtual fundamental classes \cite{BehrendFantechi} that satisfy (i) a degeneration formula \cite{Li}, which allows many computations to be reduced to toric targets, and (ii) a torus localization formula \cite{GraberPandharipande}, which reduces the Gromov--Witten theory of toric varieties to Hodge integrals. As a result, most of the known computations in Gromov--Witten theory reduce to computations of Hodge integrals. 

The dimension of a target variety bounds the degree of $\lambda$ classes that appear in the corresponding Hodge integrals. As a consequence, \emph{Result~B suggests that pseudostable Gromov--Witten theory is equal to usual Gromov--Witten theory for target curves.} Due to Proposition~\ref{thm:mumford}, we do not expect this equality to hold for targets of higher dimension, and one of the future aims of this work is to explore the relationship between pseudostable Gromov--Witten invariants and usual Gromov--Witten invariants for higher-dimensional targets.

\subsection{Acknowledgements}
	R.~Cavalieri was partially supported by a Simons Collaboration Grant (420720) and an NSF Grant (DMS-2100962). B. Van Over acknowledge the ARCS Foundation for their generous support. D.~Ross was partially supported by the San Francisco State University Presidential Award and by an NSF Grant (DMS-2001439). J.~Wise was supported by a Simons Collaboration Grant (636210) and a Simons Fellowship (822534).

\section{Comparing $\psi$ and $\lambda$ classes}\label{sec:comparingclasses}

In this section, we present a result that computes pullbacks of pseudostable $\psi$ and $\lambda$ classes from $A^*(\M_{g,n}^\ps)$ in terms of stable $\psi$  and $\lambda$ classes in $A^*(\M_{g,n})$. We begin this section with prerequisite material about the moduli spaces of (pseudo)stable curves $\M_{g,n}^{(\ps)}$ and the morphisms
\[
\T:\M_{g,n}\rightarrow\M_{g,n}^\ps.
\]
With the background material established, we then use the morphism $\T$ to prove an explicit relationship comparing the vector bundles $\T^*(\L_i)$ with $\L_i$ and $\T^*(\bE)$ with $\bE$ (Theorem \ref{thm:comparingbundles}), and then we prove an explicit relationship comparing the Chow classes $\T^*(\psi_i)$ with $\psi_i$ and $\T^*(\lambda_j)$ with $\lambda_j$ (Theorem \ref{thm:comparingclasses}).

\subsection{Background on pseudostable curves} 

The moduli spaces of pseudostable curves, denoted $\M_{g,n}^{\ps}$, were first introduced by Schubert \cite{Schubert}, and further developed by Hassett and Hyeon \cite{HH} and Federchuk and Smyth \cite{FS}. In order to describe $\M_{g,n}^\ps$ and its relationship with $\M_{g,n}$---the moduli space of stable curves---we begin by establishing conventions that will be used throughout. 

By a \emph{curve}, we mean a complete, reduced, and connected variety of dimension one. A \emph{node} is a curve singularity that is locally isomorphic to $x^2=y^2$ and a \emph{cusp} is a curve singularity that is locally isomorphic to $x^2=y^3$. The \emph{genus} of a singular curve is always taken to mean the arithmetic genus. Given a curve $C$ with only nodes and cusps as singularities, let 
\[
\eta:\widetilde C\rightarrow C
\]
be the normalization of $C$. By a \emph{normalized component} of $C$, we mean a connected component of $\widetilde C$, each of which is a smooth curve. If $C$ has marked points $p_1,\dots,p_n\in C$, then a \emph{special point} of a normalized component is any point whose image under $\eta$ is a singularity or a marked point in $C$.

We say that a curve $C$ with marked points $p_1,\dots,p_n\in C$ is \emph{stable} if
\begin{enumerate}
\item $C$ has only nodes as singularities,
\item all marked points are distinct and contained in the smooth locus of $C$,
\item every normalized component of genus zero contains at least three special points, and
\item every normalized component of genus one contains at least one special point.
\end{enumerate}
Two marked curves $(C,p_1,\dots,p_n)$ and $(D,q_1,\dots,q_n)$ are said to be \emph{isomorphic} if there is an isomorphism $f:C\rightarrow D$ such that $f(p_i)=q_i$ for all $i$, and the moduli space $\M_{g,n}$ parametrizes flat families of stable curves of genus $g$ with $n$ marked points, up to isomorphism. The study of stable curves traces its origins back to the 1960s, initiated with the foundational work of Deligne and Mumford \cite{DeligneMumford}.  For all values of $(g,n)\in\N^2$ for which $2g-2+n>0$, the moduli space $\M_{g,n}$ is a smooth proper Deligne-Mumford stack of dimension $3g-3+n$. If $(g,n)$ is equal to $(0,0)$, $(0,1)$, $(0,2)$, or $(1,0)$, then the moduli space is empty.

In the 1990s, while studying the GIT construction of the coarse underlying space of $\M_{g,n}$, Schubert introduced an alternative notion of stability \cite{Schubert} called pseudostability. We say that a curve $C$ with marked points $p_1,\dots,p_n\in C$ is \emph{pseudostable} if
\begin{enumerate}
\item $C$ has only nodes \underline{and cusps} as singularities,
\item all marked points are distinct and contained in the smooth locus of $C$,
\item every normalized component of genus zero contains at least three special points, 
\item every normalized component of genus one contains at least \underline{two} special points, and
\item \underline{every normalized component of genus two contains at least one special point}.
\end{enumerate}
For clarity, we have underlined the aspects where pseudostability differs from stability. The moduli space $\M_{g,n}^\ps$ parametrizes flat families of pseudostable curves of genus $g$ with $n$ marked points, up to isomorphism. For all values of $(g,n)$ except $(0,0)$, $(0,1)$, $(0,2)$, $(1,0)$, $(1,1)$, and $(2,0)$, the moduli space $\M_{g,n}^\ps$ is a smooth proper Deligne-Mumford stack of dimension $3g-3+n$.  If $(g,n)$ is equal to $(0,0)$, $(0,1)$, $(0,2)$, $(1,0)$, $(1,1)$, or $(2,0)$ then $\M_{g,n}^\ps$ is empty; otherwise we say that $(g,n)$ are \emph{pseudostable indices}.

Pseudostability differs from usual stability in that it allows cusps while disallowing \emph{elliptic tails}, which are 
irreducible components of genus one that do not contain any marked points and intersect the rest of the curve in a single point. For all pseudostable indices $(g,n)$, there is a morphism
\[
\T:\M_{g,n}\rightarrow\M_{g,n}^\ps.
\]
To describe this morphism on the level of points, consider a stable curve $(C,p_1,\dots,p_n)$ and the corresponding point $[(C,p_1,\dots,p_n)]\in\M_{g,n}$. Write
\[
C=C_0\cup E_1\cup\dots\cup E_r
\]
where $E_1,\dots,E_r$ are the elliptic tails of $C$. Let $(\widehat C,\hat p_1,\dots,\hat p_n)$ be the unique pseudostable curve that admits a morphism $T:C\rightarrow \widehat C$ such that
\begin{enumerate}
\item[(i)] $T$ is an isomorphism when restricted to $C\setminus (E_1\cup\dots\cup E_r)$,
\item[(ii)] $T(E_i)$ is a cusp on $\widehat C$ for all $i=1,\dots,r$, and
\item[(iii)] $T(p_i)=\hat p_i$.
\end{enumerate}
With this notation, we have
\[
\T([(C,p_1,\dots,p_n)])=[(\widehat C,\hat p_1,\dots, \hat p_n)]\in\M_{g,n}^\ps.
\]
Informally, we think of $\T$ as contracting the elliptic tails in $C$ to cusps in $\widehat C$. 

\subsection{$\psi$ and $\lambda$ classes}\label{sec:psilambda}

In order to define the pseudostable $\psi$ and $\lambda$ classes, let us start by recalling the usual definition of these classes for the moduli space of stable curves. Consider the universal curve $\pi:\C\rightarrow\M_{g,n}$, with $n$ sections  $\sigma_1,\dots,\sigma_n:\M_{g,n}\rightarrow \C$, whose images are the marked points. Let $\omega_\pi$ be the relative dualizing sheaf of $\pi$. The \emph{cotangent line bundles} on $\M_{g,n}$ are defined by
\[
\L_i=\sigma_i^*(\omega_\pi);
\]
the fiber of $\L_i$ over a point $[(C,p_1,\dots,p_n)]$ is the cotangent line of $C$ at $p_i$. The $\psi$ classes are the first Chern classes of these line bundles:
\[
\psi_i=c_1(\L_i)\in A^1(\M_{g,n})\;\;\;\text{ for }\;\;\;i=1,\dots,n.
\]
The \emph{Hodge bundle} on $\M_{g,n}$ is defined by
\[
\bE=\pi_*(\omega_\pi);
\]
its fiber over a point $[(C,p_1,\dots,p_n)]$ is the $g$-dimensional vector space of global sections of the dualizing sheaf $\omega_C$. The $\lambda$ classes are the Chern classes of this vector bundle:
\[
\lambda_j=c_j(\bE)\in A^j(\M_{g,n})\;\;\;\text{ for }\;\;\;j=0,\dots,g.
\]

For the moduli spaces of pseudostable curves, the $\lambda$ and $\psi$ classes in $A^*(\M_{g,n}^\ps)$ can be defined in exactly the same way, using instead the universal curve $\pi^{\ps}:\C^\ps\rightarrow\M_{g,n}^\ps$ and the sections $\sigma_1^\ps,\dots,\sigma_n^\ps:\M_{g,n}^\ps\rightarrow\C^\ps$.  We note that the same argument as for $\M_{g,n}$ shows that $\pi_*(\omega_{\pi^{\ps}})$ is a vector bundle over $\M_{g,n}^{\ps}$.%
\footnote{In particular, $R^2 \pi^\ps_*(\mathcal O_{\C^\ps}) = 0$, so $R^1 \pi^\ps_*(\mathcal O_{\C^\ps})$ commutes with base change and $\pi^\ps_*(\mathcal O_{\C^\ps})$ commutes with base change, so $R^1 \pi^\ps_*(\mathcal O_{\C^\ps})$ is a vector bundle.  By Serre duality, $\pi^\ps_*(\omega_{\pi^\ps})$ is also a vector bundle.}
With a goal of understanding how the numerical invariants of moduli spaces of curves depend on the stability condition used to compactify them, we aim to study of the \emph{pseudostable Hodge integrals}:
\[
\int_{\M_{g,n}^\ps}F(\lambda_1,\dots,\lambda_g,\psi_1,\dots,\psi_n)\in\Q.
\]
where $F\in\Q[x_1,\dots,x_g,y_1,\dots,y_n]$ is a polynomial and the integral denotes the pushforward to a point. In order to compute pseudostable Hodge integrals, we will translate the computation  to $\M_{g,n}$, where much is already known about Hodge integrals. To carry out this translation, notice that
\[
\int_{\M_{g,n}^\ps}F(\lambda_1,\dots,\lambda_g,\psi_1,\dots,\psi_n)=\int_{\M_{g,n}}\T^*F(\lambda_1,\dots,\lambda_g,\psi_1,\dots,\psi_n).
\]
Since $\T^*$ is a ring homomorphism, the latter integrand can be written as
\[
F(\T^*(\lambda_1),\dots,\T^*(\lambda_g),\T^*(\psi_1),\dots,\T^*(\psi_n)).
\] 
Since $\T^*(\lambda_j)=c_j(\T^*(\bE))$ and $\T^*(\psi_i)=c_1(\T^*(\L_i))$, we can reduce the problem of computing pseudostable Hodge integrals to understanding the vector bundles $\T^*(\bE)$ and $\T^*(\L_i)$ on $\M_{g,n}$, which is the aim of the next subsection.

\subsection{Comparing vector bundles}

In order to compute the pullbacks $\T^*(\bE)$ and $\T^*(\L_i)$, one might naturally expect the gluing map 
\[
\G:\M_{g-1,n+1}\times \M_{1,1}\rightarrow\M_{g,n},
\]
to come into play, because the image of $\G$ parametrizes curves with at least one elliptic tail, which is exactly where $\T$ fails to be an isomorphism. Let
\[
p_1:\M_{g-1,n+1}\times \M_{1,1}\rightarrow\M_{1,1}
\] 
be the projection onto the second factor. The next result computes $\T^*(\bE)$ and $\T^*(\L_i)$.

\begin{theorem}\label{thm:comparingbundles}
Let $(g,n)$ be  pseudostable indices.
\begin{enumerate}
\item[(i)] There is an isomorphism of line bundles on $\M_{g,n}$:
\[
\T^*(\L_i)=\L_i.
\]
\item[(ii)] There is a short exact sequence of coherent sheaves on $\M_{g,n}$:
\[
0\longrightarrow \T^*(\bE^\vee) \longrightarrow\bE^\vee\longrightarrow \G_*(p_1^*(\bE)^\vee) \longrightarrow 0.
\]
\end{enumerate}
\end{theorem}

%

\begin{proof}[Proof of Theorem \ref{thm:comparingbundles}]
Consider the universal curve $\pi:\C\rightarrow\M_{g,n}$ and let $\widehat\C=\T^*(\C^\ps)$ be the family of pseudostable curves associated to the morphism $\T:\M_{g,n}\rightarrow\M_{g,n}^\ps$. This gives rise to the following commutative diagram, which will be referenced throughout the proof:
\begin{center}
\begin{tikzcd}[column sep=small]
\C \arrow{rr}{\widehat\T}\arrow{dr}{\pi}	&			& \widehat\C\arrow{dl}[above]{\hat\pi\;\;\;}  \\
								& \M_{g,n}		&								
\end{tikzcd}
\end{center}

To prove (i), define $\hat\sigma_i:\M_{g,n}\rightarrow\widehat\C$ by $\hat\sigma_i=\T^\ast \sigma_i$ for all $i$. By definition of the pullback, 
\[
\T^*(\L_i)=\hat\sigma_i^*(\omega_{\hat\pi}).
\]
Let $\cE\subseteq\C$ be the locus of elliptic tails, and notice that $\widehat\T$ is a fiberwise isomorphism on the complement of $\cE$. Since $\hat\sigma_i=\widehat\T\circ\sigma_i$ and the image of $\sigma_i$ is contained in the complement of $\cE$, it follows that 
\[
\hat\sigma_i^*(\omega_{\hat\pi})=\sigma_i^*(\omega_\pi)=\L_i,
\]
which proves (i).

To prove (ii), consider the Grothendieck spectral sequence associated to the composition $\pi_*=\hat\pi_*\circ\widehat\T_*$:
\[
(R^p\hat\pi_*\circ R^q\widehat\T_*)(\O_\C)\Longrightarrow R^{p+q}\pi_*(\O_\C).
\]
Using that $H^2$ vanishes on curves, the exact sequence of low degrees is a short exact sequence:
\begin{equation}\label{eq:ses}
0\longrightarrow R^1\hat\pi_*(\widehat\T_*(\O_\C))\longrightarrow R^1\pi_*(\O_\C)\longrightarrow \hat\pi_*(R^1\widehat\T_*(\O_\C)) \longrightarrow 0.
\end{equation}
We claim that the short exact sequence \eqref{eq:ses} is the same one that appears in assertion (ii) of the theorem. Notice that the middle term is, by Serre duality, $\bE^\vee$. Thus, it remains to identify the first and third terms with $\T^*(\bE^\vee)$ and $\G_*(p_1^*(\bE)^\vee)$, respectively.

Analyzing the first term in \eqref{eq:ses}, we begin with the claim that $\widehat\T_*(\O_\C)=\O_{\widehat\C}$. To justify this, consider the Stein factorization of $\widehat\T:\C\rightarrow\widehat\C$:
\[
\C\stackrel{f}{\longrightarrow}\widetilde\C=\mathrm{Spec}_{\widehat\C}(\widehat\T_*(\O_\C))\stackrel{g}{\longrightarrow}\widehat\C.
\]
	By construction of the Stein factorization, $f_*(\O_\C)=\O_{\widetilde\C}$ and $g$ is finite. In this particular setting, $g$ is also birational and $\widehat\C$ is normal,%
	\footnote{Note that $\widehat\C$ is a flat family of curves over the normal base $\M_{g,n}$, so it is (S2); its generic fiber is smooth and its fibers are reduced so it is (R1).  Therefore it is normal by Serre's criterion.}
	so it follows from Zariski's Main Theorem that $g$ is an isomorphism. Thus, 
\[
\O_{\widehat \C}=g_*(\O_{\widetilde\C})=g_*f_*(\O_\C)=\widehat\T_*(\O_C),
\]
and we conclude that%
	\footnote{Since $R^2 \hat\pi_\ast \mathcal O_{\C^\ps} = 0$ on the fibers, and $\hat\pi_\ast \mathcal O_{\C^\ps} = \mathcal O_{\M^\ps_{g,n}}$ on the fibers, cohomology and base change implies $R^1 \hat\pi_\ast \mathcal O_{\widehat\C}$ is a vector bundle and commutes with base change.} 
\[
R^1\hat\pi_*(\widehat\T_*(\O_\C))=R^1\hat\pi_*(\O_{\widehat\C}).
\]
In addition, using the fact that $\widehat\C$ is the pullback along $\T:\M_{g,n}\rightarrow\M_{g,n}^\ps$ of the universal curve $\pi^{\ps}:\C^{ps}\rightarrow\M_{g,n}^\ps$, we have
\[
R^1\hat\pi_*(\O_{\widehat\C})=\T^*(R^1\pi^\ps_*(\O_{\C^\ps}))=\T^*(\bE^\vee).
\]
Thus, we conclude that the first term \eqref{eq:ses} is $\T^*(\bE^\vee)$, as desired.

Next, we analyze the third term in \eqref{eq:ses}. We begin with the claim that
\begin{equation}\label{eq:ronetails}
R^1\widehat\T_*(\O_\C)=R^1\widehat\T_*(\O_\cE).
\end{equation}
To prove this, consider the short exact sequence
\[
0\longrightarrow\O_\C(-\cE)\longrightarrow\O_\C\longrightarrow \O_\cE\longrightarrow 0.
\]
The last three nonzero terms in the corresponding long exact sequence are
\[
R^1\widehat\T_*(\O_\C(-\cE))\longrightarrow R^1\widehat\T_*(\O_\C)\longrightarrow R^1\widehat\T_*(\O_\cE)\longrightarrow 0.
\]
To prove \eqref{eq:ronetails}, we must show that $R^1\widehat\T_*(\O_\C(-\cE))=0$, and it suffices to prove that the fiber $H^1(\widehat\T^{-1}(x),\O_\C(-\cE)|_{\widehat\T^{-1}(x)})$ vanishes for any $x\in\widehat\C$. If $x$ is not a cusp, then $\widehat\T^{-1}(x)$ is a single point, so $H^1$ vanishes. If $x$ is a cusp, then $\widehat\T^{-1}(x)=E$ is an elliptic tail on some fiber and the cohomology group becomes $H^1(E,\O_\C(-\cE)|_E)=H^1(E,\O_E(q))$ where $q$ is the point where $E$ attaches to the rest of the fiber. The latter group vanishes by the Riemann-Roch Theorem, which concludes the proof of \eqref{eq:ronetails}.

Now consider the commutative diagram
\begin{center}
\begin{tikzcd}
\cE_1 \arrow{r}{\cong}\arrow{d}{\pi_1}		& \cE \arrow{r}{\widehat\T} \arrow{d}{\pi}		& \widehat\cE \arrow{dl}{\hat\pi}\\
\M_{g-1,n+1}\times\M_{1,1}\arrow{r}{\G}		& \M_{g,n}								
\end{tikzcd}
\end{center}
where $\cE_1$ is the pullback of the universal curve from $\M_{1,1}$ and $\widehat\cE\subseteq\widehat\C$ is the schematic image of $\cE \subset \widehat\C$---in other words, it is the locus of cusps in the fibers of $\widehat\C$ over $\M_{g,n}$. Notice that $\cE_1$ is isomorphic to $\cE$, as they both parametrize a point on an elliptic tail of a curve in $\M_{g,n}$. By Serre duality, 
\[
p_1^*(\bE)^\vee=R^1{\pi_1}_*(\O_{\cE_1}),
\]
and since $\G$ is finite, we have that
\[
\G_*(p_1^*(\bE)^\vee)=\G_*R^1{\pi_1}_*(\O_{\cE_1})=R^1(\G\circ\pi_1)_*(\O_{\cE_1})=R^1\pi_*(\O_\cE).
\]
The second equality is because $\G_\ast$ is exact (since $\G$ is finite). Similarly, since $\widehat\pi:\hat\cE\rightarrow\M_{g,n}$ is finite, we have
\[
R^1\pi_*(\O_\cE)=R^1(\hat\pi\circ\widehat\T)_*(\O_\cE)=\hat\pi_*(R^1\widehat\T_*(\O_\cE)),
\]
which, by \eqref{eq:ronetails}, is equal to the third term of \eqref{eq:ses}, completing the proof of (ii).
\end{proof}

\subsection{Comparing Chern classes}

Now that we have an understanding of the vector bundles $\T^*(\bE)$ and $\T^*(\L_i)$, we can compute their Chern classes. In this section, we prove an explicit formula for the pullbacks of pseudostable $\psi$ and $\lambda$ classes.

In order to state the formula, let
\[
\G^k:\M_{g-k,n+k}\times\M_{1,1}^{\times k}\rightarrow \M_{g,n}
\]
be the gluing map onto the locus of $k$ elliptic tails and let $p_0$ and $p_1$ denote the projection maps onto the first two factors of the domain. Notice that, when $k=1$, we recover $\G^1=\G$, and our notation is consistent with the map $p_1$ used in the previous section. The next result computes explicit formulas for $\T^*(\psi_i)$ and $\T^*(\lambda_j)$ in terms of these maps.

\begin{theorem}\label{thm:comparingclasses}
For all pseudostable indices $(g,n)$, we have
\[
\T^*(\psi_i)=\psi_i\;\;\;\text{ for all }\;\;\;i=1,\dots,n
\]
and
\[
\T^*(\lambda_j)=\lambda_j+\sum_{i=1}^j\frac{1}{i!}\G_*^{i}(p_0^*(\lambda_{j-i})).
\]
\end{theorem}

\begin{proof}

The first identity is immediate from Theorem \ref{thm:comparingbundles}(i). The proof of the second identity is more involved. To investigate the Chern classes $\T^*(\lambda_j)=c_j(\T^*(\bE))$, we begin by studying the Chern characters $\ch_j(\T^*(\bE))$. Applying Chern characters to the short exact sequence in Theorem~\ref{thm:comparingbundles}(ii), we obtain
\begin{equation}\label{eq:cherncharacters}
\ch(\T^*(\bE^\vee))=\ch(\bE^\vee)-\ch(\G_*(p_1^*(\bE)^\vee)).
\end{equation}
Using that $\ch_j(V^\vee)=(-1)^j\ch_j(V)$, the relation \eqref{eq:cherncharacters} implies that
\begin{equation}\label{eq:cherncharacter1}
\ch_j(\T^*(\bE))=\ch_j(\bE)-(-1)^j\ch_j(\G_*(p_1^*(\bE)^\vee)),
\end{equation}
We now aim to compute the final term $\ch(\G_*(p_1^*(\bE)^\vee))$. To do so, we apply the Grothendieck--Riemann--Roch Theorem:
\[
\ch(\G_*(p_1^*(\bE)^\vee))=\G_*(\ch(p_1^*(\bE)^\vee)\td(T_\G)),
\]
where $\td(T_\G)$ is the Todd class of the relative tangent sheaf of $\G$. 

To compute $\td(T_\G)$, we note that the relative tangent sheaf is the $K$-theoretic additive inverse of the pullback of the normal bundle of the image of $\G:\M_{g-1,n+1}\times\M_{1,1}\rightarrow\M_{g,n}$, which is
\[
p_0^*(\L_{n+1}^\vee)\otimes p_1^*(\L_1^\vee),
\]
Thus, by the multiplicativity of the Todd class, we compute
\[
\td(T_\G)=\td\big(p_0^*(\L_{n+1}^\vee)\otimes p_1^*(\L_1^\vee)\big)^{-1}=\frac{e^{\psi_\bullet+\psi_\star}-1}{\psi_\bullet+\psi_\star}, 
\]
where, for simplicity, we define
\[
\psi_\star=p_0^*(\psi_{n+1})\;\;\;\text{ and }\;\;\;\psi_\bullet=p_1^*(\psi_1).
\]

To compute $\ch(p_1^*(\bE)^\vee)$, we note that, on the one-dimensional moduli space $\M_{1,1}$, there is a natural isomorphism of line bundles $\bE=\L_1$, so
\[
\ch(\bE)=e^{\lambda_1}=1+\lambda_1=1+\psi_1.
\]
Thus, 
\[
\ch(p_1^*(\bE)^\vee)=p_1^*(\ch(\bE^\vee))=p_1^*(1-\psi_1)=1-\psi_\bullet.
\]

Repeatedly using the fact that $\psi_\bullet^2=0$, we simplify as follows:
\begin{align*}
\ch(p_1^*(\bE)^\vee)\td(T_\G)&=(1-\psi_\bullet)\frac{e^{\psi_\bullet+\psi_\star}-1}{\psi_\bullet+\psi_\star}\\
&=(1-\psi_\bullet)\sum_{k\geq 0}\frac{(\psi_\bullet+\psi_\star)^k}{(k+1)!}\\
&=(1-\psi_\bullet)\left(\sum_{k\geq 0}\frac{\psi_\star^k+k\psi_\bullet\psi_\star^{k-1}}{(k+1)!} \right)\\
&=\sum_{k\geq 0}\frac{\psi_\star^k}{(k+1)!}-\sum_{k\geq 0}\frac{\psi_\bullet\psi_\star^k}{(k+1)!}+\sum_{k\geq 0}\frac{k\psi_\bullet\psi_\star^{k-1}}{(k+1)!}\\
&=\sum_{k\geq 0}\frac{\psi_\star^k}{(k+1)!} - \sum_{k\geq 0} \frac{\psi_\bullet \psi_\star^{k-1}}{k!} + \sum_{k\geq 0} \frac{k \psi_\bullet \psi_\star^{k-1}}{(k+1)!} \\
&=\sum_{k\geq 0}\frac{\psi_\star^k-\psi_\bullet\psi_\star^{k-1}}{(k+1)!}.
\end{align*}
Starting with the third line, we have used the convention that any $\psi$ class to a negative power is equal to zero.

Using that $\G$ has relative dimension $-1$, we see that
\[
\ch_j(\G_*(p_1^*(\bE)^\vee))=\G_*\left(\frac{\psi_\star^{j-1}-\psi_\bullet\psi_\star^{j-2}}{j!}\right).
\]
Putting this formula for $\ch_j(\G_*(p_1^*(\bE)^\vee))$ back into Equation~\eqref{eq:cherncharacter1}, we conclude that
\begin{equation}\label{eq:cherncharactercomparison}
\ch_j(\T^*(\bE))=\ch_j(\bE)-\frac{(-1)^j}{j!}\G_*\left(\psi_\star^{j-1}-\psi_\bullet\psi_\star^{j-2}\right).
\end{equation}

In order to translate the comparison in Equation~\eqref{eq:cherncharactercomparison} to a statement about Chern classes, instead of Chern characters, we utilize Bell polynomials (see \cite[Section 3.3]{Comtet}, for an introduction to Bell polynomials). The Bell polynomials $B_n(x)\in\Z[x_1,\dots,x_n]$ can be defined by the series expansion
\[
\exp\bigg(\sum_{j=0}^\infty t^j\frac{x^j}{j!}\bigg)=\sum_{n=0}^\infty B_n(x)\frac{t^n}{n!}.
\]
They satisfy a number of useful properties; we list the two that are most relevant.
\begin{enumerate}
\item[(B1)] If $e_i$ is the degree-$i$ elementary symmetric polynomial in some set of variables and $p_i$ is the degree-$i$ power sum polynomial in the same variables, then 
\[
e_j=\frac{1}{j!}B_j(0!p_1,\;-1!p_2,\;2!p_3,-3!p_4,\;\dots,\;(-1)^{j-1}(j-1)!p_j).
\]
\item[(B2)] Bell polynomials are determined recursively: $B_0(x)=1$ and
\[
B_{k+1}(x)=\sum_{j=0}^{k}{k\choose j}x_{j+1}B_{k-j}(x),
\]
\end{enumerate}

Notice that (B1) gives us a way to represent Chern classes, which are elementary symmetric functions in the Chern roots, in terms of Chern characters, which (up to scalar factor) are power sum functions in the Chern roots. To make this precise in our setting, let $\rho_1,\dots,\rho_g$ denote the Chern roots of $\bE$. Then
\[
\lambda_j=e_j(\rho_1,\dots,\rho_g)\;\;\;\text{ and }\;\;\;\ch_j(\bE)=\frac{1}{j!}p_j(\rho_1,\dots,\rho_g).
\]
Thus, if we define
\[
x=(x_\ell)_{\ell=1}^g=\big((-1)^{\ell-1}(\ell-1)!\ell!\ch_\ell(\bE)\big)_{\ell=1}^g, 
\]
it follows from (B1) that
\[
\lambda_j=\frac{1}{j!}B_j(x).
\]
Defining another sequence of variables by 
\[
\displaystyle y=\big((\ell-1)!\G_*(\psi_\star^{\ell-1}-\psi_\bullet\psi_\star^{\ell-2})\big)_{\ell=1}^g, 
\]
the Chern character comparison of Equation~\eqref{eq:cherncharactercomparison} implies that
\[
x+y=((-1)^{\ell-1}(\ell-1)!\ell!\ch_\ell(\T^*(\bE)))_{\ell=1}^g,
\]
from which it follows that
\[
\T^*(\lambda_j)=\frac{1}{j!}B_j(x+y).
\]
Thus, the formula for $\T^*(\lambda_j)$ in the statement of the theorem is equivalent to 
\begin{equation}\label{eq:recursionprelim}
B_j(x+y)=j!\bigg(\lambda_j+\sum_{i=1}^j\frac{\G_*^{i}(p_0^*(\lambda_{j-i}))}{i!}\bigg).
\end{equation}

We can finish the proof of the theorem by showing that the right-hand side in \eqref{eq:recursionprelim} satisfies the Bell polynomial recursion (B2). Carefully translating this recursion, it remains to prove the following identity for all $k\geq 0$:
\begin{equation}\label{eq:recursion1}
(k+1)\bigg(\lambda_{k+1}+\sum_{i=1}^{k+1}\frac{\G_*^{i}(p_0^*(\lambda_{k+1-i}))}{i!}\bigg)=\sum_{j=0}^k\frac{x_{j+1}+y_{j+1}}{j!}\bigg(\lambda_{k-j}+\sum_{i=1}^{k-j}\frac{\G_*^{i}(p_0^*(\lambda_{k-j-i}))}{i!}\bigg).
\end{equation}
Using the Bell polynomial recursion (B2), it follows that
\begin{equation}\label{eq:brforlambda}
(k+1)\lambda_{k+1}=\sum_{j=0}^k\frac{x_{j+1}}{j!}\lambda_{k-j}.
\end{equation}
Therefore, we can reduce \eqref{eq:recursion1} to the following equation
\begin{equation}\label{eq:recursion2}
(k+1)\sum_{i=1}^{k+1}\frac{\G_*^{i}(p_0^*(\lambda_{k+1-i}))}{i!}=\sum_{j=0}^k\frac{1}{j!}\bigg(y_{j+1}\lambda_{k-j}+(x_{j+1}+y_{j+1})\sum_{i=1}^{k-j}\frac{\G_*^{i}(p_0^*(\lambda_{k-j-i}))}{i!}\bigg).
\end{equation}
We now compute the three types of products appearing in the right-hand side of \eqref{eq:recursion2}:
\begin{enumerate}
\item $y_{j+1}\lambda_{k-j}$,
\item $x_{j+1}\G_*^{i}(p_0^*(\lambda_{k-j-i}))$, and
\item $y_{j+1}\G_*^{i}(p_0^*(\lambda_{k-j-i}))$.
\end{enumerate}

To compute the Type~(1) terms, we start by unpacking the definitions:
\[
y_{j+1}\lambda_{k-j}=j!\G_*(\psi_\star^j-\psi_\bullet\psi_\star^{j-1})\lambda_{k-j}.
\]
Noting that the Hodge bundle splits on the boundary:
\[
\G^*(\bE)=p_0^*(\bE)\oplus p_1^*(\bE),
\]
the projection formula then implies that
\begin{align*}
y_{j+1}\lambda_{k-j}&=j!\G_*\big((\psi_\star^j-\psi_\bullet\psi_\star^{j-1})(p_0^*(\lambda_{k-j})+p_1^*(\lambda_1)p_0^*(\lambda_{k-j-1})\big)\\
&=j!\G_*\big(\psi_\star^jp_0^*(\lambda_{k-j})-\psi_\bullet\psi_\star^{j-1}p_0^*(\lambda_{k-j})+\psi_\bullet\psi_\star^jp_0^*(\lambda_{k-j-1})\big),
\end{align*}
where the second equality uses the facts that $p_1^*(\lambda_1)=\psi_\bullet$ and $\psi_\bullet^2=0$. We also note the convention that $\lambda_{-1}=0$. Adding over all $j$, and canceling the telescoping summands, we see that the contribution of the terms of Type~(1) to the right-hand side of \eqref{eq:recursion2} is equal to
\begin{equation}\label{eq:recursiontype1}
\sum_{j=0}^k\frac{y_{j+1}\lambda_{k-j}}{j!}=\sum_{j=0}^k\G_*\big(\psi_\star^jp_0^*(\lambda_{k-j})\big).
\end{equation}

To compute the terms of Type (2), we begin by unpacking the definitions:
\[
x_{j+1}\G_*^{i}(p_0^*(\lambda_{k-j-i}))=(-1)^{j}j!(j+1)!\ch_{j+1}(\bE)\G_*^{i}(p_0^*(\lambda_{k-j-i})).
\]
To write a formula for these terms, we generalize the definition of $\psi_\bullet$ from $i=1$ to $i\geq 1$ using the same definition $\psi_\bullet=p_1^*(\psi_1)=p_1^*(\lambda_1)$. It then follows from the projection formula, the splitting of the Hodge bundle, and the vanishing of $p_\ell^*(\lambda_{j+1})$ for $j > 0$, that
\[
	\begin{aligned}
		x_{j+1}\G_*^{i}(p_0^*(\lambda_{k-j-i}))&=
(-1)^j j! (j+1)! \G_*^i\Bigl(p_0^*(\lambda_{k-j-i})\sum_{\ell=1}^i p_\ell^\ast(\lambda_{j+1})\Bigr) 
		\\ &=
\begin{cases}
i\G_*^i(\psi_\bullet p_0^*(\lambda_{k-i}))+\G_*^i(p_0^*(\ch_{1}(\bE)\lambda_{k-i})) & j=0,\\
(-1)^jj!(j+1)!\G_*^i(p_0^*(\ch_{j+1}(\bE)\lambda_{k-j-i})) & j >  0.
\end{cases}
	\end{aligned}
\]
Note that $\G_*^i(p_\ell^*(\lambda_1)p_0^*(\lambda_{k-i}))$ all coincide, so we have (slightly abusively) written $\G_*^i(\psi_\bullet p_0^*(\lambda_{k-i})$ for their common value.  If we add these contributions over $i$ and $j$ and use the fact that
\[
\sum_{j=0}^{k-i}(-1)^j(j+1)!\ch_{j+1}(\bE)\lambda_{k-i-j}=(k-i+1)\lambda_{k-i+1},
\]
which is equivalent to \eqref{eq:brforlambda}, we see that the contribution of the terms of Type~(2) to the right-hand side of \eqref{eq:recursion2} is equal to
\begin{align}
	\notag& \sum_{j=0}^k\sum_{i=1}^{k-j}\frac{x_{j+1}\G_*^{i}(p_0^*(\lambda_{k-j-i}))}{j!i!}
	\\ \notag & \qquad =\sum_{i=1}^k\frac{\G_*^i(\psi_\bullet p_0^*(\lambda_{k-i}))}{(i-1)!}+\sum_{i=1}^k \frac{1}{i!} \sum_{j=0}^{k-i} (-1)^j (j+1)! \G_*^i(p_0^*(\ch_{j+1}(\bE)\lambda_{k-j-i})
	\\ \notag & \qquad =\sum_{i=1}^k\frac{\G_*^i(\psi_\bullet p_0^*(\lambda_{k-i}))}{(i-1)!}+\sum_{i=1}^{k}\frac{k-i+1}{i!}\G_*^i(p_0^*(\lambda_{k-i+1}))
	\\ \label{eq:recursiontype2} & \qquad =\sum_{i=1}^k\frac{\G_*^i(\psi_\bullet p_0^*(\lambda_{k-i}))}{(i-1)!}+(k+1) \sum_{i=1}^{k}\frac{1}{i!}\G_*^i(p_0^*(\lambda_{k-i+1}))-\sum_{i=0}^{k-1}\frac{1}{i!} \G_*^{i+1}(p_0^*(\lambda_{k-i}))
\end{align}
In the last line, we have rewritten the expression to make eventual cancellation more obvious.  

Lastly, to compute the terms of Type~(3), we start by unpacking definitions:
\[
y_{j+1}\G_*^{i}(p_0^*(\lambda_{k-j-i}))=j!\G_*(\psi_\star^{j}-\psi_\bullet\psi_\star^{j-1})\G_*^{i}(p_0^*(\lambda_{k-j-i})).
\]
To compute the product of the pushforward classes, we can use Graber and Pandharipande's formula for the intersection of boundary strata \cite[Appendix~A.4, Equation~(11)]{GP} (see also \cite[Proposition~1]{Yang}).  Here we are intersecting the boundary stratum $\G^i$ with $\G = \G^1$.  The intersection has $i$ non-transverse copies of $\G^i$, each of which contributes weighted by the first Chern class of the excess bundle, which is $-\psi_\bullet-\psi_\star$.  The intersection also has $1$ transverse copy of $\G^{i+1}$.  On this copy, the $p_0^*(\lambda_{k-j-i})$ factor is pulled back from $\M_{g-1,n+1} \times \M_{1,1}$ and therefore factors into $p_1^*(\lambda_1) p_0^*(\lambda_{k-j-i-1}) + p_0^*(\lambda_{k-j-i})$, where we are abusing $p_0$ and asking it now to stand for the projection to $\M_{g-1,n}$, and $\lambda_{-1}=0$ by convention. We have
\begin{align*}
	\frac{1}{j!}y_{j+1}\G_*^{i}(p_0^*(\lambda_{k-j-i}))&=i\G_*^i\big((\psi_\star^{j}-\psi_\bullet\psi_\star^{j-1})(-\psi_\bullet-\psi_\star)p_0^*(\lambda_{k-j-i})\big)\\
	&\qquad+\G^{i+1}_*\big((\psi_\star^{j}-\psi_\bullet\psi_\star^{j-1})(p_1^*(\lambda_1)p_0^*(\lambda_{k-j-i-1})+p_0^*(\lambda_{k-j-i}))\big)\\
	&= \Bigl\{ - i\G^i_*(\psi_\bullet p_0^*(\lambda_{k-i})) \quad \text{if $j=0$} \Bigr\}
		\\ & \qquad - i \G_*^i(\psi_\star^{j+1} p_0^*(\lambda_{k-j-i})) + \G_*^{i+1}(\psi_\star^j p_0^*(\lambda_{k-j-i})) 
		\\ & \qquad + \G_*^{i+1}(\psi_\bullet\psi_\star^jp_0(\lambda_{k-j-i-1})) - \G_*^{i+1}(\psi_\bullet\psi_\star^{j-1}p_0^*(\lambda_{k-j-i})).
\end{align*}
We have simplified using $p_1^*(\lambda_1)=\psi_\bullet$ and $\psi_\bullet^2=0$ (so $(\psi_\star^j-\psi_\bullet\psi_\star^{j-1})(-\psi_\bullet-\psi_\star)=-\psi_\star^{j+1}$ for $j>0$ and $(\psi_\star^j-\psi_\bullet\psi_\star^{j-1})p_1^*(\lambda_1) = \psi_\star^j\psi_\bullet$).  

When we sum over $j$, the terms in the bottom line will cancel.  When we sum over all $i$ and $j$, we therefore have
\begin{multline*}
	\sum_{\substack{i\geq1 \\ j\geq 0 \\ i+j\leq k}} \frac{1}{i!j!} y_{j+1}\G^i_*(p_0^*(\lambda_{k-j-i})) 
 = - \sum_{i=1}^k \frac{1}{(i-1)!} \G^i_*(\psi_\bullet p_0^*(\lambda_{k-i})) \\
	+ \sum_{\substack{i\geq1 \\ j\geq 0 \\ i+j\leq k}} \Bigl( - \frac{1}{(i-1)!} \G_*^i(\psi_\star^{j+1}p_0^*(\lambda_{k-j-i})) + \frac{1}{i!} \G_*^{i+1}(\psi^j_\star p_0^*(\lambda_{k-j-i})) \Bigr)
\end{multline*}
The last two terms telescope to leave only the first term with $i= 1$ and the last with $j= 0$):
\begin{align}
	 \notag& \sum_{j=0}^k\sum_{i=1}^{k-j}\frac{y_{j+1}\G_*^{i}(p_0^*(\lambda_{k-j-i}))}{j!i!} 
	\\ \notag & \qquad =-\sum_{i=1}^k\frac{\G_*^i(\psi_\bullet p_0^*(\lambda_{k-i}))}{(i-1)!}
		-\sum_{j=0}^{k-1}\G_*\big(\psi_\star^{j+1}p_0^*(\lambda_{k-j-1})\big)
		+\sum_{i=1}^{k}\frac{\G^{i+1}_*\big(p_0^*(\lambda_{k-i})\big)}{i!}
	\\\label{eq:recursiontype3}  & \qquad=-\sum_{i=1}^k\frac{\G_*^i(\psi_\bullet p_0^*(\lambda_{k-i}))}{(i-1)!}
	-\sum_{j=1}^{k}\G_*\big(\psi_\star^{j}p_0^*(\lambda_{k-j})\big)
	+\sum_{i=1}^{k-1}\frac{\G^{i+1}_*\big(p_0^*(\lambda_{k-i})\big)}{i!} 
	\\ \notag
	& \omit\hfill $\displaystyle + (k+1) \frac{\G^{k+1}(p_0^*(\lambda_{k-i}))}{(k+1)!}$
\end{align}

Again, we have rewritten the formula in the last line to make eventual cancellation clearer.  Now we add the contributions of terms of Types (1), (2), and (3), which are recorded in the right-hand sides of Equations \eqref{eq:recursiontype1}, \eqref{eq:recursiontype2}, and \eqref{eq:recursiontype3}, to get the right-hand side of Equation~\eqref{eq:recursion2}.  The first term of~\eqref{eq:recursiontype3} cancels with the first term of~\eqref{eq:recursiontype2}; the second and third terms of~\eqref{eq:recursiontype3} cancel with~\eqref{eq:recursiontype1} and the third term of~\eqref{eq:recursiontype2}; the second term of~\eqref{eq:recursiontype2} and the last term of~\eqref{eq:recursiontype3} remain and combine to give the left-hand side of~\eqref{eq:recursion2}, which finishes the proof of the theorem.
\end{proof}

\section{Computing pseudostable Hodge integrals}

It follows from Theorem \ref{thm:comparingclasses} and the discussion in Subsection~\ref{sec:psilambda} that pseudostable Hodge integrals can be translated to intersection numbers on moduli spaces of stable curves. For ease of reference, we now state this result precisely.

\begin{theorem}\label{thm:translation}
For any psuedostable indices $(g,n)$ and polynomial $F\in\Q[x_1,\dots,x_g,y_1,\dots,y_n]$, we have
\[
\int_{\M_{g,n}^\ps}F(\lambda_1,\dots,\lambda_g,\psi_1,\dots,\psi_n)=\int_{\M_{g,n}}F(\hat\lambda_1,\dots,\hat\lambda_g,\psi_1,\dots,\psi_n)
\]
where
\[
\hat\lambda_j=\lambda_j+\sum_{i=1}^j\frac{1}{i!}\G^i_*\big(p_0^*(\lambda_{j-i})\big).
\]
\end{theorem}

We now employ Theorem~\ref{thm:translation} to carry out the first computations of pseudostable Hodge integrals. 

\subsection{Linear Hodge integrals}

It follows immediately from Theorem~\ref{thm:translation} that any psuedostable integral of $\psi$ classes alone is equal to the corresponding stable integral of $\psi$ classes. Once $\lambda$ classes make an appearance, however, you wouldn't expect the integrals to remain unchanged. However, if the expression is linear in the $\lambda$ classes, it turns out that the pseudostable Hodge integral is exactly the same as the corresponding stable Hodge integral.

\begin{proposition}\label{thm:linearhodge}
For any $j=1,\dots,g$ and any polynomial $F\in\Z[x_1,\dots,x_n]$,
\[
\int_{\M_{g,n}^\ps}\lambda_jF(\psi_1,\dots,\psi_n)=\int_{\M_{g,n}}\lambda_jF(\psi_1,\dots,\psi_n).
\]
\end{proposition}

\begin{proof}
We can assume that $F$ is homogenous of degree $3g-3+n-j$, as both integrals vanish for homogeneous polynomials of any other degree. By Theorem~\ref{thm:translation}, we have
\[
 \int_{\M_{g,n}^\ps}\lambda_jF(\psi_1,\dots,\psi_n)=\int_{\M_{g,n}}\lambda_jF(\psi_1,\dots,\psi_n)+\sum_{i=1}^j\int_{\M_{g,n}}\frac{\G^i_*\big(p_0^*(\lambda_{j-i}\big)}{i!}F(\psi_1,\dots,\psi_n).
\]
Noting that $(\G^i)^*(F(\psi_1,\dots,\psi_n))=p_0^*(F(\psi_1,\dots,\psi_n))$, it follows from the projection formula that 
\[
\G^i_*\big(p_0^*(\lambda_{j-i})\big)F(\psi_1,\dots,\psi_n)=\G^i_*\big(p_0^*(\lambda_{j-i}F(\psi_1,\dots,\psi_n))\big).
\]
However,
\[
\lambda_{j-i}F(\psi_1,\dots,\psi_n)=0\in A^*(\M_{g-i,n+i})
\]
by dimension reasons---the class has degree $3g-3+n-i$ while the dimension of the moduli space is $3g-3+n-2i$.
\end{proof}

\begin{remark}
Proposition~\ref{thm:linearhodge} was first discovered by the second named author as part of his Master's thesis. At the time, we were very surprised to discover this result because we had not yet found an explicit formula for $\T^*(\lambda_j)$---we were working directly with the Chern character comparison of Equation~\eqref{eq:cherncharactercomparison} and the cancellation of the correction terms in the linear Hodge integrals seemed magical at the time. The discovery of Proposition~\ref{thm:linearhodge} then motivated the quest to find a more concise formula for $\T^*(\lambda_j)$, leading to Theorem~\ref{thm:comparingclasses}.
\end{remark}

\begin{remark}\label{rmk:ELSV}
We expect that Proposition~\ref{thm:linearhodge} also admits a more conceptual proof using ideas related to the ELSV formula. The ELSV formula, named after Ekedahl, Lando, Shapiro, and Vainshtein \cite{ELSV}, is a formula relating Hurwitz numbers to linear Hodge integrals:
\begin{equation}\label{eq:ELSV}
h^m_{\mu_1,\dots,\mu_\ell}=m!\prod_{i=1}^\ell \frac{\mu_i^{\mu_i+1}}{\mu_i!}\int_{\mathcal{\overline{M}}_{g,\ell}}\frac{1-\lambda_1+\lambda_2-\cdots+(-1)^g\lambda_g}{(1-\mu_1\psi_1)\cdots(1-\mu_\ell\psi_\ell)}
\end{equation}
where $h^m_{\mu_1,\dots,m_\ell}$ counts the number of ways to factor a permutation $\mu\in S_n$ of cycle type $(\mu_1,\dots,\mu_\ell)$ into a product of $m$ transpositions that act transitively on $S_n$. By varying $\mu_1,\dots,\mu_\ell$, the ELSV formula determines all linear Hodge integrals.

There are several proofs of the ELSV formula, but the one that is most pertinent to this discussion is the one given by Graber and Vakil \cite{GP}. In their proof, they consider the moduli space of relative stable maps to $\mathbb{P}^1$ and interpret Hurwitz numbers in terms of the degree of the branch morphism. They then show that the ELSV formula \eqref{eq:ELSV} arises upon applying the virtual localization formula to compute the degree of the branch morphism. We expect that their arguments carry over verbatim to the psuedostable setting, which would tell us that both stable and pseudostable linear Hodge integrals are determined by the same relation \eqref{eq:ELSV}, so they must be equal. Such a proof of Proposition~\ref{thm:linearhodge} would bypass the explicit comparison of Theorem \ref{thm:comparingclasses}, but it would require one to undertake the technical work of constructing moduli spaces of relative pseudostable maps and proving the corresponding virtual localization theorem, which we view as a worthwhile research endeavor, but not one that we will pursue here.
\end{remark}

\subsection{Mumford's formula}

Given Proposition~\ref{thm:linearhodge}, one might optimistically hope that all pseudostable Hodge integrals are equal to their stable counterparts. As it turns out, this is far too optimistic. Even for Hodge integrals with a product of two lambda classes, the two types of Hodge integrals differ in general. We verify this below, while at the same time exhibiting that Mumford's formula fails in the pseudostable setting.

We recall that ``Mumford's formula'', proved in \cite{Mumford}, says that
\[
(1+\lambda_1+\cdots+\lambda_g)(1-\lambda_1+\cdots+(-1)^g\lambda_g)=1\in A^*(\M_{g,n}).
\]
Mumford's formula is a very useful tool in Gromov--Witten theory. Looking at the degree-two part of the identity, we see that $2\lambda_2-\lambda_1^2=0\in A^2(\M_{g,n})$. Therefore, in order to show that Mumford's formula fails in the pseudostable setting, it suffices to prove that
\[
\int_{\M_{g,n}^{\ps}}(2\lambda_2-\lambda_1^2)\psi_1^{3g-4}\neq 0.
\]
We prove this for all $g\geq 2$ and $n\geq 1$ in the next result.

\begin{proposition}\label{thm:mumford}
For all $n\geq 1$, we have
\[
\sum_{g\geq 2}t^g\int_{\M_{g,n}^{\ps}}(2\lambda_2-\lambda_1^2)\psi_1^{3g-5+n}=-\frac{t}{24}(e^{\frac{t}{24}}-1).
\]
In particular,
\begin{itemize}
\item for every $g\geq 2$ and $n\geq 1$, Mumford's formula does not hold in $A^*(\M_{g,n}^\ps)$, and
\item pseudostable Hodge integrals are not always equal to their stable counterparts.
\end{itemize}
\end{proposition}

\begin{proof}
By Theorem~\ref{thm:translation}, we see that
\[
\int_{\M_{g,n}^{\ps}}(2\lambda_2-\lambda_1^2)\psi_1^{3g-5+n}=\int_{\M_{g,n}}(2\hat\lambda_2-\hat\lambda_1^2)\psi_1^{3g-5+n}
\]
where
\[
\hat\lambda_1=\lambda_1+\G^1_*(1)\;\;\;\text{ and }\;\;\;\hat\lambda_2=\lambda_2+\G_*^1(p_0^*(\lambda_1))+\frac{1}{2}\G_*^2(1).
\]
Notice that
\begin{align*}
\hat\lambda_1^2&=\lambda_1^2+2\lambda_1\G^1_*(1)+\G_*^1(-\psi_{\star}-\psi_{\bullet})+\G^2_*(1)\\
&=\lambda_1^2+2\G^1_*(p_0^*(\lambda_1))+\G^1_*(\psi_\bullet)-\G^1_*(\psi_\star)+\G^2_*(1).
\end{align*}
Therefore, using the fact that $2\lambda_2-\lambda_1^2=0$ in $A^*(\M_{g,n})$, the integrand can be simplified as
\[
(-\G^1_*(\psi_{\bullet})+\G^1_*(\psi_{\star}))\psi_1^{3g-5+n}=-\G^1_*(\psi_{\bullet}p_0^*(\psi_1^{3g-5+n}))+\G^1_*(\psi_{\star}p_0^*(\psi_1^{3g-5+n})).
\]
By dimension reasons, only the first summand contributes to the integral, and we have
\begin{align*}
\int_{\M_{g,n}^{\ps}}(2\lambda_2-\lambda_1^2)\psi_1^{3g-5+n}&=-\int_{\M_{g,n}}\G^1_*(p_1^*(\psi_1)p_0^*(\psi_1^{3g-5+n}))\\
&=-\int_{\M_{1,1}}\psi_1\int_{\M_{g-1,n+1}}\psi_1^{3g-5+n}\\
&=-\int_{\M_{1,1}}\psi_1\int_{\M_{g-1,1}}\psi_1^{3g-5},
\end{align*}
where the last equality was an application of the string equation. Finally, using the Witten-Kontsevich Theorem \cite{Witten,Kontsevich}, we have the following formula for one-pointed psi class intersection numbers (see, for example, Section 3.5.5 of \cite{Kock} for a derivation):
\[
\int_{\M_{g,1}}\psi_1^{3g-2}=\frac{1}{24^gg!}.
\]
Thus, we conclude that
\[
\int_{\M_{g,n}^{\ps}}(2\lambda_2-\lambda_1^2)\psi_1^{3g-5+n}=-\int_{\M_{1,1}}\psi_1\int_{\M_{g-1,1}}\psi_1^{3g-5}=\frac{-1}{24^g(g-1)!},
\]
and the theorem follows from the expression of the exponential function as a power series.
\end{proof}

\bibliographystyle{alpha}
\bibliography{references}

\end{document}